\documentclass[a4paper]{article}
\usepackage[english]{babel}
\usepackage[utf8x]{inputenc}
\usepackage[colorinlistoftodos]{todonotes}
\usepackage{amsmath,amsfonts,amssymb,hyperref,color,graphicx,siunitx,fancyhdr,lastpage,enumerate, commath, amsthm, subcaption}
\newtheorem{theorem}{Theorem}

\newtheorem{proposition}[theorem]{Proposition}

\newtheorem{corollary}[theorem]{Corollary}

\theoremstyle{remark}

\numberwithin{theorem}{section}

\usepackage{url}
\usepackage[numbers]{natbib}

\newcommand{\R}{\mathbb R}
\newcommand{\Rd}{\mathbb R^d}
\newcommand{\E}{\mathbb E}
\renewcommand{\P}{\mathbb P}

\newcommand{\taud}{\tau_\partial}

\newcommand{\Res}{ \mathcal R}

\newcommand{\footremember}[2]{%
    \footnote{#2}
    \newcounter{#1}
    \setcounter{#1}{\value{footnote}}%
}
\newcommand{\footrecall}[1]{%
    \footnotemark[\value{#1}]%
}

\title{A note on the jump locations of Markov processes}
\author{%
Andi Q. Wang\footremember{Oxf}{University of Oxford}\footnote{Corresponding author. Address: Department of Statistics, 24-29 St Giles, Oxford, OX1 3LB, UK. Email:  \href{mailto:a.wang@stats.ox.ac.uk}{a.wang@stats.ox.ac.uk}.}
\and David Steinsaltz\footrecall{Oxf}}

\date{August 2019}

\begin{document}

\maketitle
\begin{abstract} 
For a continuous-time Markov process, we characterize the law of the first jump location when started from an arbitrary initial distribution, in terms of the invariant distribution of an auxiliary `resurrected' Markov process. This could be of interest in the burgeoning fields of piecewise-deterministic Markov chain Monte Carlo methods and quasi-stationary Monte Carlo methods.

\end{abstract}

\section{Introduction}
Consider a continuous-time c\`adl\`ag Markov process $Y=(Y_t)_{t\ge 0}$ evolving on a Polish space $\mathcal X$ with a distinguished strict measurable subset $\partial \subsetneq \mathcal X$. We write $E:= \mathcal X \backslash \partial$, and derive the distribution of the process the moment before entering $\partial$ for the first time. We assume that the process enters $\partial$ exclusively through \textit{jumps}, namely we can construct a locally bounded function $\kappa: E\to [0,\infty)$, where $\kappa(x)$ dictates the rate of transfer from $x \in E$ to $\partial$. In other words, we can express the first hitting time, $\taud$, of the process into $\partial$ as
\begin{equation}
    \taud = \inf\left \{t\ge 0: \int_0^t \kappa(Y_s)\dif s\ge \xi\right\},
\label{eq:taud}
\end{equation}
where $\xi \sim \text{Exp}(1)$ and is independent of $Y$. Of course, this $\taud$ is a stopping
time, that is, a Markov
random time. $\taud$ can be interpreted as the first arrival time of an inhomogeneous Poisson process with (stochastic) rate function $t\mapsto \kappa(Y_t)$.

We assume that for any starting point, jumping in finite time is certain; that is, for any $x\in E$,
\begin{equation*}
    \P_x (\taud <\infty) = 1.
\end{equation*}

Since we are interested only in the behaviour of the process up until $\taud$, without loss of generality we can assume that $\partial$ is a cemetery (absorbing) state, and imagine $\kappa$ to be a killing rate. In this work we will make connections with the established theory of \textit{quasi-stationarity}, the asymptotic behaviour of such killed Markov processes, conditional on survival. For an introduction to this area, see \citet{Collet2013}.


Recall that a probability distribution $\pi$ on $E$ is \textit{quasi-stationary} if for all $t\ge 0$, 
\begin{equation}
    \P_\pi (Y_t \in \cdot \,|\tau_\partial >t) = \pi(\cdot),
\label{eq:QSD}
\end{equation}
where $\P_\pi$ denotes the law of the process $Y$ with initial distribution $Y_0\sim\pi$.

In this note we characterize the distribution of the location at which the process was killed, that is, the distribution of $Y_{\taud-}$.
We offer an interpretation for the general case in terms of the invariant distribution of an auxiliary `resurrected' Markov process, which returns to the state space after being killed. 

\sloppy The recent developments in computational statistics of \textit{piecewise-deterministic Markov chain Monte Carlo} (PD-MCMC) methods and \textit{quasi-stationary Monte Carlo} (QSMC) methods, for instance \citet{Vanetti2017}, \citet{Fearnhead2018} and \citet{Pollock2016}, have particularly motivated the study of such Markov processes. 
For these methods, the aim is to simulate a continuous-time stochastic process whose behaviour alters at the arrival times of an inhomogeneous Poisson process precisely as in \eqref{eq:taud}. Practically, the simulation of these arrival times is typically performed using Poisson thinning: points are drawn from some dominating Poisson process, and these are accepted or rejected depending on the true killing rate at those times. (Even in cases where the true killing rate is potentially unbounded this can still be done through localization techniques; see for instance \citet{Pollock2016}.)
Implementations of PD-MCMC and QSMC methods will thus typically include the limiting locations $Y_{\taud -}$ at the times \eqref{eq:taud} as a free by-product. Given that we have access to these points, we may expect to be able to use them for statistical inference, once we understand
how their distribution is related to the the process that generates them. This particular application will be discussed further in Section \ref{subsec:MC_meth}.

The laws of jump locations and resurrected processes have been considered in many works previously, generally in discrete spaces, see for instance the recent \citet{Kuntz2019} and references therein and \citet[Section 4.4.1]{Collet2013}. We will provide some links to related work in Section \ref{subsec:related}. Our contribution here is to give a simple derivation for an analogous result in this present general state-space setting, connecting related notions of exit locations, quasi-stationarity and resurrected processes, in a form amenable for the recent developments in continuous-time Monte Carlo as mentioned above.

\section{Mathematical background}
We assume that we have a c\`adl\`ag \textit{right process} $\tilde Y=(\tilde\Omega, \tilde{\mathcal F}, \tilde{\mathcal F}_t, \tilde Y_t, \tilde\theta_t, \tilde \P_x)$ evolving in continuous time on a Polish space $ E$. We assume that under $\tilde \P_x$, $Y$ is \textit{unkilled}; $\tilde \P_x$ is heuristically `the law of $Y$ without jumps out of $E$'. We will augment the process with jumps out of $E$. 
The technical definition of a right process can be found in \citet[Chapter 20]{Sharpe1988}. Intuitively, right processes are an abstract class of strong Markov processes which contains most right-continuous Markov processes of practical interest, such as diffusions, L\' evy processes, Feller processes and deterministic right-continuous flows (see \citet[Exercise (8.8) and Chapter 9]{Sharpe1988}).

We take $\tilde Y$ to be the canonical realisation of the process; namely $\tilde \Omega$ is the space of c\`adl\`ag paths mapping $[0,\infty) \to E$ equipped with the cylinder $\sigma$-algebra. The shift maps $(\tilde \theta_t)$ are then defined in the usual way.

We now define an augmented process that evolves like $Y$, including jumps out of $E$, but which records the location at which $Y$ jumped. Let $\mathcal X := \{0,\partial\} \times E$. $\Omega$ is taken to be the space of c\`adl\`ag paths $[0,\infty)\to \mathcal X$ which admit $\{\partial\}\times E$ as a trap: $\omega(s)=\omega(\taud)$ for all $s\ge \taud(\omega)$, where $\taud := \inf\{t\ge0: Z_t(\omega) \in \{\partial\}\times E\}$, $Z_t(\omega)=\omega(t) \in \mathcal X$ being the coordinate process. We will write $Z_t = (e_t, Y_t)$, with $e_t \in \{0,\partial\}$ and $Y_t \in E$. $\Omega$ is equipped with the cylinder $\sigma$-algebra, and the shift maps $(\theta_t)$ are defined on $\Omega$ in the usual way.

On $\Omega$ we define killing operators $k_t: \Omega \to \Omega$ for each $t \ge 0$ as follows. The use of such killing operators was initiated by \citet{Azema1973}. Given $\omega = (e,\tilde \omega) \in \Omega$, a path $[0,\infty)\to \mathcal X$, define a new path $k_t(\omega)\in \Omega$ by
\begin{equation*}
        k_t(\omega)(s) = 
    \begin{cases}
            \omega(s)    & \text{if } t < \taud(\omega), 0\le s <t, \\
            (\partial, \tilde \omega(t))    & \text{if } t <  \taud(\omega),\quad s \ge t,\\
            \omega(s)    & \text{if } t \ge  \taud(\omega), \quad s\ge 0.
    \end{cases}
\end{equation*}
Intuitively, $k_t$ takes a path $\omega$ and `kills it at time $t$', sending it instantaneously to the trap, provided it hasn't already been killed prior to $t$.

Consider the map $\psi: \tilde \Omega \to \Omega, \tilde \omega \mapsto \psi(\tilde \omega)$, with $\psi(\tilde\omega)(s) = (0, \tilde \omega(s))$ for each $s\ge 0$. The unkilled laws $\P_x^0$ on $\mathcal X$ are defined to be the images of the laws $\tilde \P_x$ under $\psi$. Under $\P_x^0$, the trajectories, which are defined on the augmented state space $\mathcal X$, are almost surely unkilled (never entering $\{\partial\}\times E$).

Now to define the precise killing mechanism, we make use of the multiplicative functional formulation of killing as described in \citet[Chapter 61]{Sharpe1988}. We briefly review this here. Recall that a nonnegative stochastic process $(m_t)$ is a \textit{multiplicative functional} (MF) provided that for each $s, t \ge 0$,
\begin{equation*}
    \P^0_x\left (\omega: m_{t+s}(\omega) \neq m_t(\omega) m_s (\theta_t \omega)\right) = 0.
\end{equation*}
In this work we will be concerned with the $\P_x^0$-almost surely continuous decreasing right MF defined for each $t\ge 0$ by
\begin{equation*}
    m_t := \exp\left (-\int_0^t \kappa(Y_s)\dif s \right ).
\end{equation*}
Almost sure continuity of the paths $t \mapsto m_t$ follows from the fact that $Y$ is c\`adl\`ag and $\kappa$ is locally bounded. Since $Y$ is c\`adl\`ag the corresponding random measure $(-\dif m_t)$ on $[0, \infty)$ satisfies
\begin{equation*}
    -\dif m_t = \kappa(Y_t) \exp\left (-\int_0^t \kappa(Y_s)\dif s \right ) \dif t,
\end{equation*}
$\P_x^0$-almost surely.

Now we can define a new law $\P_x$ on $\mathcal X$, under which $Y$ is killed at rate $\kappa$, corresponding to the MF $(m_t)$. As in formula (61.2) of \citet{Sharpe1988}, for a bounded measurable function $H$ on $\Omega$, define for each $x \in \{0\}\times E$, $\P_x$ via
\begin{equation}
    \P_x (H) := \P_x^0 \int_0^\infty H \circ k_t\, (-\dif m_t).
    \label{eq:def_PH}
\end{equation}
 There is no mass at $\infty$, since by assumption killing happens almost surely in finite time.
Under $\P_x$, the process $Y$ is killed at rate precisely $-\dif m_t/m_t=\kappa(Y_t)$ (\citet[Exercise (61.9)]{Sharpe1988}), and its lifetime is $\P_x$-almost surely finite by assumption. We are then interested in the law of $Y_{\taud}$ under $\P_x$, where $x \in \{0\}\times E$.
From formula (61.3) of \citet{Sharpe1988} 
we have the relation
\begin{equation}
    \P_x\left (H\, 1_{\{t< \taud\}}\right ) = \P_x^0 (H \,m_t).
    \label{eq:MF_ident}
\end{equation}
for nonnegative measurable functions $H$ on $\mathcal F_t$.

Until recently, specific study of Markov processes with soft killing, as in \eqref{eq:taud}, have been relatively neglected in the literature on quasi-stationarity. Key contributions in this area for the continuous-state-space setting have been \citet{Steinsaltz2007}, \citet{Kolb2012b}, \citet[Section 4.4]{Champagnat2017}, \citet{Velleret2018} and the recent work on QSMC methods, \citet{Pollock2016} and \citet{Wang2019}. The results of these works all function in the diffusion
setting relevant to QSMC methods, which are the main area of application 
anticipated for the present work.

Within the quasi-stationarity literature, questions about the exit locations have also received little attention. One exception is the following elegant result.
\begin{theorem}
(\citet[Proposition 2]{Martinez2008}, repeated in \citet[Theorem 2.6]{Collet2013}.) Let $\pi$ be a probability measure on $\{0\}\times E$ which is a quasi-stationary distribution for the process $Z$. Then $\taud$ and $Y_{\taud}$ are $\P_\pi$-independent random variables.
\label{thm:2.1}
\end{theorem}

\noindent \citet{Martinez2008} and \citet{Collet2013} derive the relation
\begin{equation}
    \frac{\dif}{\dif t} \P_\pi(Z_t\in \{\partial\} \times A)\bigg|_{t=0} = \theta(\pi)\P_\pi(Z_{\taud }\in \{\partial\}\times A) = \theta(\pi)\P_\pi(Y_{\taud}\in A)
\label{eq:pi_Y_T_A}
\end{equation}
for any measurable set $A\subset E$, where $\theta(\pi)$ in this case is given by
\begin{equation*}
    \theta(\pi) = \int_E \kappa(x)\pi(\dif x)<\infty.
\end{equation*}
$\theta(\pi)$ must necessarily be finite, since $\pi$ otherwise could
not be a quasi-stationary distribution (see, for instance, Theorem 2.2 of \citet{Collet2013}). $\theta(\pi)$ is the quasi-stationary killing rate: we have
\begin{equation}
    \P_\pi (\taud > t) = e^{-\theta(\pi)t}\quad \forall t \ge 0.
\label{eq:QSD_killing_exp}
\end{equation}
In particular, \eqref{eq:pi_Y_T_A} allows us to characterize the distribution of $Y_{\taud}$ under $\P_\pi$, which by Theorem \ref{thm:2.1} is independent of the time $\taud$.

\begin{proposition}
For any measurable $A\subset E$,
\begin{equation*}
    \P_\pi(Y_{\taud}\in A) = \frac{\int_A \kappa(x)\pi(\dif x) }{\theta(\pi)}.
\end{equation*}
\label{Prop:Y_T_pi}
\end{proposition}

\noindent This states that the law of $Y_{\taud}$ under $\P_\pi$ is (proportional to) $\kappa(x)\pi(\dif x)$.\vspace{\baselineskip}

\begin{proof}
     Heuristically, this follows from the relation \eqref{eq:pi_Y_T_A} and the definition of $\taud$ in \eqref{eq:taud}. Formally, it follows 
     as a special case of our main result Theorem \ref{Thm:Y_taud_gen}.
\end{proof}

The goal of this note is to characterize the distribution of $Y_{\taud}$ under $\P_\mu$, when started in some arbitrary initial distribution $\mu$ on $\{0\}\times E$.
We write $P_t(x,\dif y)$ for the sub-Markovian transition kernel of the killed process $Y$ under $\P_x$, which defines a semigroup. That is, for $x \in E$, $f: E \to \R$,
\begin{equation*}
    P_t(x,f)=\int_E f(y)P_t(x,\dif y)=\P_x\left [f(Y_t) \, 1_{\{\taud > t\}}\right ]
\end{equation*}
whenever this integral makes sense.

Recall that the resolvent operator (at 0) $\Res$, mapping nonnegative measurable functions to nonnegative measurable functions, is defined by
\begin{equation}
    \Res f(x) = \int_0^\infty P_t(x,f)\dif t.
\label{eq:Res_def}
\end{equation}
This is the Green's function, as described 
for instance in \citet{Dynkin1969}.
Using \eqref{eq:MF_ident} and exchanging the order of integration (by Tonelli's Theorem),
\begin{equation*}
    \Res \kappa(x) = \P_x^0 \left[\int_0^\infty \kappa(Y_s) m_s\dif s\right] = \P_x^0\left[\int_0^\infty (-\dif m_s) \right]=1.
\end{equation*}


Given a probability measure $\mu$ on $E$ we consider the $\mu$-{resurrected process}. The $\mu$-resurrected process is a c\`adl\`ag (unkilled) Markov process $X=(X_t)_{t\ge 0}$ evolving on $E$. This process evolves according to the law of the killed process $Y$, except that at killing events, which similarly occur at rate $t\mapsto \kappa(X_t)$, the location is resampled according to the fixed measure $\mu$; it then evolves independently from there. This process can be carefully constructed explicitly via the techniques of \citet[Chapter II.14]{Sharpe1988} on concatenated processes, where it is shown that the resulting concatenated process is also a right process.

Such processes have long been associated with the study of quasi-stationarity, {\em cf.}, for instance \citet{Bartlett1960, Darroch1965} and more recently \citet{Barbour2010, Barbour2012}. They were utilised to particularly great effect in \citet{Ferrari1995}, where they were used to prove the existence of quasi-stationary distributions for discrete-state-space continuous-time Markov chains. They have also been applied to discrete time on general (compact) spaces in \citet{Benaim2018} and continuous spaces in \citet{Wang2018AnApprox}. This process is analogous to the immediate-return procedure of \citet{Doob1945}, going back to the very foundations of the field itself. 
 Resurrecting processes are also considered in \citet{Pakes1993}, to classify recurrence of states for Markov chains. The invariant measure of such a process will provide an interpretation of our main result. Here we consider the resurrected process in continuous time on general state spaces.

When generators are available (for instance in \citet{Wang2018AnApprox} for diffusions on compact manifolds), the generator of the resurrecting process $X$ can be expressed as
\begin{equation*}
    L_X f(x)= L^0 f(x)+ \kappa (x)\int \left(f(y)-f(x)\right)\mu(\dif y)
\end{equation*}
for functions $f$ in the appropriate domain, where $L^0$ denotes the infinitesimal generator of the unkilled Markov process $\tilde Y$.


It has been demonstrated in various recent works that an invariant distribution for the resurrecting process of certain processes is given by $\Pi(\mu)$, where
\begin{equation*}
    \Pi(\mu)(f)\propto \mu \Res f=\int_E \mu(\dif x) \Res(f)(x) ,
\end{equation*}
provided that this is integrable, namely that
\begin{equation}
    \mu\Res 1 = \E_\mu[\taud]<\infty.
    \label{eq:taud_integrable}
\end{equation}
See, for instance, \citet{Collet2013, Benaim2018, Wang2018AnApprox}. Here is a general theorem, couched in terms of general regenerative processes, which possess regenerations times $T_1<T_2<\dots$ at which times the process `starts afresh'; see \citet{Asmussen2003} for the precise definitions. For our $\mu$-resurrecting process, the times at which the process is reborn according to $\mu$ are of course regeneration times.

\begin{theorem}[Theorem~1.2, \citet{Asmussen2003}]
    Assume that an undelayed regenerative process $(X_t)$ has a metric state space, right-continuous paths and nonlattice inter-regeneration times with finite mean. Then the limiting distribution $\nu$ exists and is given by
    \begin{equation*}
        \nu(f) = \frac{\E\left[\int_0^{T_1} f(X_s)\dif s\right]}{\E[T_1]}.
    \end{equation*}
\end{theorem}
This expression in our setting is precisely equivalent to $\Pi(\mu)$.
So for our $\mu$-resurrected process, provided that \eqref{eq:taud_integrable} holds and $\taud$ has finite mean under $\E_\mu$ and is nonlattice, i.e. not concentrated on a set of the form $\{\delta, 2\delta, \dots\}$, $\Pi(\mu)= \mu\Res(\cdot)/\mu\Res 1$ will define its invariant distribution. 

\section{Main Result}
\begin{theorem}
For any measurable $A\subset E$,
\begin{equation*}
    \P_\mu (Y_{\taud} \in A) = \mu\Res(\kappa 1_A).
\end{equation*}
\label{Thm:Y_taud_gen}
\end{theorem}
\noindent Given the preceding discussion, this immediately implies the following interpretation. 
\begin{corollary}
When the probability distribution $\Pi(\mu)$ exists --- that is, when \eqref{eq:taud_integrable} holds --- for any measurable $A\subset E$,
 \begin{equation*} 
     \P_\mu (Y_{\taud} \in A) = \frac {\int_A \kappa(x) \Pi(\mu)(\dif x)}{\int_E \kappa(x) \Pi(\mu)(\dif x)}.
 \end{equation*}
\label{cor:Pi_mu}
\end{corollary}
\noindent That is, the law of $Y_{\taud}$ under $\P_\mu$ is (proportional to) $\kappa(\cdot) \Pi(\mu)$; \textit{cf.} Proposition \ref{Prop:Y_T_pi}.\vspace{\baselineskip}

\begin{proof}[Proof of Theorem \ref{Thm:Y_taud_gen}]
    Since $\P_\mu = \int \mu(\dif x)\P_x$, it suffices to prove the result for point masses, for $x\in E$. We are concerned with the probability of
    $\{Y_{\taud}\in A\}$ under $\P_x$. By \eqref{eq:def_PH}
    \begin{align*}
    \P_x\bigl( Y_{\taud}\in A \bigr) &= \P_x^0 \int_0^\infty 1_A\left (Y_{\taud}\left (k_t(\omega)\right)\right )(-\dif m_t)\\
        &= \P_x^0 \int_0^\infty 1_A(Y_{t}) \,\kappa(Y_t) m_t \dif t \\
        &= \int_0^\infty  \P_x^0 \left[ 1_A(Y_{t})\kappa(Y_t) m_t \right] \dif t\\
        &= \int_0^\infty \P_x [1_A(Y_{t})\kappa(Y_t) 1_{\{\taud >t\}}] \dif t\\
        &= \int_0^\infty P_t(x, \kappa 1_A) \dif t\\
        &= \Res(\kappa 1_A)(x).
    \end{align*}
    In the first line we used the definition of the measure $\P_x$, \eqref{eq:def_PH}. We then make use of Tonelli's theorem to exchange the order of integration and the identity \eqref{eq:MF_ident}. 
\end{proof}

\section{Examples and remarks}
\subsection{Related work}
\label{subsec:related}
Related questions concerning exit locations have been considered in many places. We review some of these here.

An analogue of Theorem \ref{Thm:Y_taud_gen} appears in \citet[Section 6]{Dynkin1969} 
in the context of Martin boundary theory for the discrete-time, discrete state space
setting. Dynkin considered the last
exist of a Markov chain $(X_n)_{n\in \mathbb N}$, with transition probabilities $p(n,x,y) := \P_x(X_n = y)$ from a distinguished subset $D$. Setting
$\tau := \sup\{n: x_n \in D\}$,
he noted that
\begin{equation}
    \P_x (X_\tau = y) = \sum_{m=0}^\infty p(m,x,y) \P_y(\tau = 0).
    \label{eq:discrete-time}
\end{equation}
This is analogous to our operator $\Res$ in \eqref{eq:Res_def}.

\citet{Dynkin1969b} went on to explore, in a very general continuous-time Markov process setting,
the intimate connection between the distribution of exit locations of Markov processes and \textit{excessive functions} (generalisations of nonnegative superharmonic functions). Here Dynkin rigorously derives the existence of exit locations as a limit in the \textit{Martin compactum} -- a compact space in which the original space is densely embedded, on which the \textit{Martin function} is continuously extended. Heuristically, the Martin function is a scaled version of the Green's function,
and the fundamental properties that
Dynkin describes include a characterisation of the space of `admissible' exit locations in terms of the Martin function. 
The paper explores the fundamental
theoretical connections with the Green's kernel --- what we have called the resolvent ---
but does not concern itself with an explicit or practically applicable 
description of the exit distributions.

Versions of Theorem \ref{Thm:Y_taud_gen} for discrete spaces in continuous time have also been given in various places, for example in \citet{Syski1992} and recently in \citet{Kuntz2019}, where an elegant algorithm to approximate related distributions is also formulated. 


Our contribution in this work is a simple derivation of an analogous continuous-time continuous-state-space result to \eqref{eq:discrete-time}, connecting the related notions of exit locations, quasi-stationarity and the invariant distributions of resurrected processes, in a form amenable for the recent developments in continous-time Monte Carlo methods.  
Such PDMP and quasi-stationary Monte Carlo methods rely upon the same continuous state space setting with inhomogeneous Poisson clock that we have presented here.


\subsection{General observations}
The quantity $\mu \Res(\kappa 1_A)$ appearing in Theorem \ref{Thm:Y_taud_gen} has a natural interpretation. Since
\begin{equation*}
    \mu\Res(\kappa 1_A) = \E_\mu\left[\int_0^{\taud} \kappa(Y_t) 1\{Y_t \in A\}\dif t\right],
\end{equation*}
we can think of this as the `average amount of killing picked up by the process in set $A$ when started from $\mu$'. Since the average total killing picked up (that is, when $A=E$) is 1, this indeed will correspond to the probability of being killed in $A$.

Theorem \ref{Thm:Y_taud_gen} is valid even in situations where there is \textit{no} quasi-stationary distribution. For example, consider the case of a continuous-time simple symmetric random walk on $\mathbb Z$, where at position $i\in \mathbb Z$, transitions to states $i-1$ and $i+1$ occur at rate 1, and there are no other transitions. In addition, we have a uniform killing rate of $\kappa(i)=1$ for all $i \in \mathbb Z$. Clearly there can be no quasi-stationary distribution, since conditioning on survival reverts us to the simple symmetric random walk on $\mathbb Z$, which has no stationary distribution. Let us take the initial distribution to be $\mu=\delta_0$. Theorem \ref{Thm:Y_taud_gen} is still valid, and tells us that the distribution of the location at which the particle is killed is the same as the invariant distribution of the simple symmetric random walk on $\mathbb Z$ which also has additional jumps to 0 at a uniform rate 1. This process has a unique invariant distribution, as it is uniformly ergodic in the sense of \citet{Down1995}. The invariant distribution can be computed exactly, through routine calculations involving the $Q$-matrix.

An example where Theorem \ref{Thm:Y_taud_gen} is still applicable even when $\Pi(\mu)$ is not well-defined (so Corollary \ref{cor:Pi_mu} does not apply) is a continuous-time simple symmetric random walk on $\mathbb Z$ as before, except that killing occurs \textit{only} at a finite collection of states $\{1,2,\dots,k\}$, with at least some of the killing
rates $\kappa_i$, $i=1, \dots, k$ nonzero. We take $\mu=\delta_0$. Since the unkilled process is recurrent, we know that killing will occur almost surely at a finite time, from any initial position. However, since the process is \textit{null} recurrent, the expectation of the return times to a given state are infinite. In particular this implies that $\E_\mu[\taud]=\infty$. In this case Theorem \ref{Thm:Y_taud_gen} still holds, hence the distribution of the killing point is the stationary distribution, reweighted by $\kappa$, of the
Markov process on $\{1,\dots,k\}$, with transition rates
$q_{ij} = 1$ if $|i-j| = 1$, except for $q_{21}=1+\kappa_2$, and $q_{i1} = \kappa_i$ for $i\ge 3$.

 

\subsection{Quasi-stationary case}
Consider the situation where there is a quasi-stationary distribution $\pi$ as in \eqref{eq:QSD}.

In this situation, Theorem \ref{Thm:Y_taud_gen} implies Proposition \ref{Prop:Y_T_pi}, since $\pi$ being quasi-stationary implies that there exists $\theta(\pi)$ such that for all non-negative measurable $f$,
\begin{equation*}
     \int \pi(\dif x) P_t(x,f)=  e^{-\theta(\pi)t} \pi(f).
\end{equation*}
This implies that that $\Pi(\pi)=\pi$. Indeed, in many cases $\Pi(\mu)=\mu$ is necessary and sufficient for $\mu$ to be quasi-stationary (see Lemma 4.2 of \citet{Benaim2018}, Proposition 2.9 of \citet{Wang2018AnApprox},  \citet[Section 4.4.1]{Collet2013}).

Furthermore, if in addition $\kappa$ has a positive uniform lower bound
${\kappa(x)\ge \epsilon>0}$ for all $x \in E$, our result can be interpreted as follows: The distribution of $Y_{\taud}$ under $\P_\pi$ can be written as
\begin{equation}
    \frac{\theta'(\pi)}{\theta'(\pi)+\epsilon}\frac{(\kappa(x)-\epsilon) \pi(\dif x)}{\theta'(\pi)} + \frac{\epsilon}{\theta'(\pi)+\epsilon}\pi (\dif x),
\label{eq:mixture}
\end{equation}
where $\theta'(\pi) := \int_E (\kappa(x)-\epsilon) \pi(\dif x) = \theta(\pi)-\epsilon$. The expression \eqref{eq:mixture} is saying that the distribution of the exit location is a mixture of the quasi-stationary distribution $\pi$ and the modified exit location for the process killed at rate $\kappa - \epsilon$. Since the killing time under $\P_\pi$ is exponential with rate $\theta(\pi)$, see \eqref{eq:QSD_killing_exp}, this can be exactly seen as a mixture corresponding to competing independent exponential clocks, at rates $\theta'(\pi)$ and $\epsilon$. This 
observation may have useful practical consequences when implementing QSMC methods, since these often take $\kappa$ to have such a uniform positive lower bound.


\subsection{Application to Monte Carlo methods}
\label{subsec:MC_meth}
As mentioned in the introduction, this note is motivated by
novel Monte Carlo methods that rely upon the simulation of continuous-time Markov processes, with randomly timed jumps defined as in \eqref{eq:taud}. Examples include the Bouncy Particle Sampler (BPS) of \citet{Bouchard-Cote2018}, the Zig-Zag process of \citet{Bierkens2019} and the ScaLE algorithm of \citet{Pollock2016}.
The former two Markov processes evolve deterministically on $\Rd$, except
the velocity jumps at random event times selected as in \eqref{eq:taud}. For the ScaLE algorithm of \citet{Pollock2016}, the dynamics are simple Brownian motion on $\Rd$, and at killing times the particles are re-weighted.

It is carefully noted in \citet[Section 6.1]{Bierkens2019} that the distribution of the Zig-Zag process at the switching times (that is, precisely the distribution of $Y_{\taud}$) is not equal to the invariant distribution, but rather is `biased towards the tails of the target distribution'. Our work here confirms that, as one would expect, this bias is precisely the event rate.

Practically speaking, the greatest challenge in simulating such processes is often the difficulty of selecting the locations at which the Markov process jumps. 
This is commonly done by simulating the jump time $\taud$ as in \eqref{eq:taud} via Poisson thinning, and then given the time $\taud$, simulating the corresponding location $Y_{\taud-}$. Our result provides a direct characterisation of the distribution of $Y_{\taud-}$. Regardless of the specific implementation, practitioners will typically have access to the values $Y_{\taud-}$.

For example, in the Zig-Zag and BPS the process moves along rays at a constant velocity. For a given fixed start point and fixed initial direction, by reparameterising we can imagine that the particle moves along $\mathcal X =[0,\infty)$ with fixed velocity $+1$. For a given locally bounded event rate $\kappa: \mathcal X \to [0,\infty)$, arrival events are defined as in \eqref{eq:taud}. Since we are starting from a fixed location, in terms of the resurrecting process, we take the rebirth distribution $\mu$ simply to be $\delta_0$. 

Simple calculations with the generator of this resurrecting process show that its invariant distribution is given by $\Pi(\mu)(\dif x)\propto \exp(-\int_0^x \kappa(z) \dif z)\dif x$, when this is integrable and sufficient regularity properties hold on $\kappa$ (to ensure that use of the generator is valid; such conditions have been formulated in \citet{Durmus2018}). This expression is consistent with what is given in the supplementary material of \cite{Bouchard-Cote2018} for the BPS. In the context of PD-MCMC -- given the specific choice of the event rate $\kappa$, which is typically minus the derivative of the log-density -- this formula will recover the form of the target distribution restricted to the ray.


Finally, we speculate that new Monte Carlo methods could be designed
to exploit the understanding of the \textit{exact} distribution of the exit locations offered by our result.

\section*{Acknowledgements}
A. Q. Wang would like to thank M. Pollock for suggesting this question and P. J. Fitzsimmons for suggesting the use of killing operators. Research of A. Q. Wang is supported by EPSRC OxWaSP CDT through grant EP/L016710/1.

\bibliography{exit_loc}
\bibliographystyle{plainnat}

\end{document}